\documentclass[proceedings,submission]{dmtcs}

\usepackage[utf8x]{inputenc}
\usepackage[english]{babel}
\usepackage{amsmath, amsfonts, amssymb}
\usepackage{subfigure}
\usepackage{cite}
\usepackage{tikz}

\author{Samuele Giraudo\addressmark{1}}
\title{Constructing combinatorial operads from monoids}
\address{\addressmark{1}Institut Gaspard Monge, Université Paris-Est
    Marne-la-Vallée, 5 Boulevard Descartes, Champs-sur-Marne,
    77454 Marne-la-Vallée cedex 2, France}

\newtheorem{Theoreme}{Theorem}[section]
\newtheorem{Proposition}[Theoreme]{Proposition}

\newcommand{\EnsPermu}{\mathfrak{S}}
\newcommand{\EnsNat}{\mathbb{N}}
\newcommand{\K}{\mathbb{K}}
\newcommand{\CalP}{\mathcal{P}}
\newcommand{\CalQ}{\mathcal{Q}}
\newcommand{\T}{{\sf T}}
\newcommand{\Unite}{{\bf 1}}
\newcommand{\End}{{\it End}}
\newcommand{\FP}{{\it PF}}
\newcommand{\MT}{{\it PW}}
\newcommand{\Per}{{\it Per}}
\newcommand{\APE}{{\it PRT}}
\newcommand{\FCat}[1]{{\it FCat^{(#1)}}}
\newcommand{\Schr}{{\it Schr}}
\newcommand{\Motz}{{\it Motz}}
\newcommand{\Comp}{{\it Comp}}
\newcommand{\AnD}{{\it DA}}
\newcommand{\SComp}{{\it SComp}}
\newcommand{\DD}{{\it D}}
\newcommand{\Dias}{{\it Dias}}
\newcommand{\NAP}{{\it NAP}}
\newcommand{\La}{{\tt a}}
\newcommand{\Lb}{{\tt b}}
\newcommand{\Lc}{{\tt c}}

\newcommand{\FCatOpA}{\scalebox{.12}{
\begin{tikzpicture}
    \node[NoeudDyck](1)at(0,0){};
    \node[NoeudDyck](2)at(1,1){};
    \node[NoeudDyck](3)at(2,2){};
    \node[NoeudDyck](4)at(3,1){};
    \node[NoeudDyck](5)at(4,0){};
    \draw[PasDyck](1)--(2);
    \draw[PasDyck](2)--(3);
    \draw[PasDyck](3)--(4);
    \draw[PasDyck](4)--(5);
\end{tikzpicture}}\,}

\newcommand{\FCatOpB}{\scalebox{.12}{
\begin{tikzpicture}
    \node[NoeudDyck](1)at(0,0){};
    \node[NoeudDyck](2)at(1,1){};
    \node[NoeudDyck](3)at(2,0){};
    \node[NoeudDyck](4)at(3,1){};
    \node[NoeudDyck](5)at(4,0){};
    \draw[PasDyck](1)--(2);
    \draw[PasDyck](2)--(3);
    \draw[PasDyck](3)--(4);
    \draw[PasDyck](4)--(5);
\end{tikzpicture}}\,}

\newcommand{\SchrOpA}{\scalebox{.18}{
\begin{tikzpicture}
    \node[NoeudSr](1)at(0,0){};
    \node[Feuille](2)at(-1,-1){};
    \node[Feuille](3)at(0,-1){};
    \node[Feuille](4)at(1,-1){};
    \draw[Arete](1)--(2);
    \draw[Arete](1)--(3);
    \draw[Arete](1)--(4);
\end{tikzpicture}}\,}

\newcommand{\SchrOpB}{\scalebox{.18}{
\begin{tikzpicture}
    \node[NoeudSr](1)at(0,0){};
    \node[Feuille](2)at(-1,-1){};
    \node[NoeudSr](3)at(.5,-.5){};
    \node[Feuille](4)at(0,-1){};
    \node[Feuille](5)at(1,-1){};
    \draw[Arete](1)--(2);
    \draw[Arete](1)--(3);
    \draw[Arete](3)--(4);
    \draw[Arete](3)--(5);
\end{tikzpicture}}\,}

\newcommand{\SchrOpC}{\scalebox{.18}{
\begin{tikzpicture}
    \node[NoeudSr](1)at(0,0){};
    \node[NoeudSr](2)at(-.5,-.5){};
    \node[Feuille](3)at(-1,-1){};
    \node[Feuille](4)at(0,-1){};
    \node[Feuille](5)at(1,-1){};
    \draw[Arete](1)--(2);
    \draw[Arete](2)--(3);
    \draw[Arete](2)--(4);
    \draw[Arete](1)--(5);
\end{tikzpicture}}\,}

\newcommand{\MotzOpA}{\scalebox{.18}{
\begin{tikzpicture}
    \node[NoeudDyck](1)at(0,0){};
    \node[NoeudDyck](2)at(1,0){};
    \draw[PasDyck](1)--(2);
\end{tikzpicture}}\,}

\newcommand{\MotzOpB}{\scalebox{.18}{
\begin{tikzpicture}
    \node[NoeudDyck](1)at(0,0){};
    \node[NoeudDyck](2)at(1,1){};
    \node[NoeudDyck](3)at(2,0){};
    \draw[PasDyck](1)--(2);
    \draw[PasDyck](2)--(3);
\end{tikzpicture}}\,}

\newcommand{\CompOpA}{\scalebox{.15}{
\begin{tikzpicture}
    \node[Boite]at(0,0){};
    \node[Boite]at(1,0){};
\end{tikzpicture}}\,}

\newcommand{\CompOpB}{\scalebox{.15}{
\begin{tikzpicture}
    \node[Boite2]at(0,0){};
    \node[Boite2]at(0,-1){};
\end{tikzpicture}}\,}

\definecolor{Noir}{RGB}{0,0,0}
\definecolor{Rouge}{RGB}{205,35,38}
\definecolor{Bleu}{RGB}{2,60,195}
\definecolor{Vert}{RGB}{23,163,1}
\definecolor{Violet}{RGB}{181,18,225}
\definecolor{Orange}{RGB}{255,113,15}

\tikzstyle{Noeud} = [circle, draw = Bleu!100, fill = Bleu!25, thick, inner sep = 0pt, minimum size = 10mm, font = \Huge]
\tikzstyle{NoeudSr} = [Noeud, minimum size = 4mm]
\tikzstyle{Arete} = [Rouge!80, thick, draw, line width = 2pt]
\tikzstyle{Feuille} = [rectangle, draw = Noir!100, fill = Noir!30, thick, inner sep = 0pt, minimum size = 3mm]
\tikzstyle{Marque1} = [draw = Vert!100, fill = Vert!80]
\tikzstyle{Marque2} = [draw = Orange!100, fill = Orange!60]
\tikzstyle{NoeudDyck} = [circle, draw = Violet!90, fill = Bleu!60, thick, inner sep = 0pt, minimum size = 5mm]
\tikzstyle{PasDyck} = [color=Violet!60,line width=3pt]
\tikzstyle{Grille} = [color=Noir!30]
\tikzstyle{Boite} = [rectangle, draw = Noir!100, fill = Vert!40, thick, inner sep = 0pt, minimum size = 10mm]
\tikzstyle{Boite2} = [rectangle, draw = Noir!100, fill = Orange!10, thick, inner sep = 0pt, minimum size = 10mm]
\tikzstyle{Boite3} = [rectangle, draw = Noir!100, fill = Rouge!10, thick, inner sep = 0pt, minimum size = 10mm]
\tikzstyle{Injection} = [Noir!100,draw,>->]
\tikzstyle{Surjection} = [Noir!100,draw,->>]

\begin{document}

\maketitle

\begin{abstract}
    \paragraph{Abstract.}
    We introduce a functorial construction which, from a monoid, produces
    a set-operad. We obtain new (symmetric or not) operads as suboperads
    or quotients of the operad obtained from the additive monoid. These
    involve various familiar combinatorial objects: parking functions,
    packed words, planar rooted trees, generalized Dyck paths, Schröder trees,
    Motzkin paths, integer compositions, directed animals,~\emph{etc.} We also
    retrieve some known operads: the magmatic operad, the commutative associative
    operad, and the diassociative operad.

    \paragraph{Résumé.}
    Nous introduisons une construction fonctorielle qui, à partir d'un monoïde,
    produit une opérade ensembliste. Nous obtenons de nouvelles opérades
    (symétriques ou non) comme sous-opérades ou quotients de l'opérade
    obtenue à partir du monoïde additif. Celles-ci mettent en jeu divers
    objets combinatoires familiers~: fonctions de parking, mots tassés,
    arbres plans enracinés, chemins de Dyck généralisés, arbres de Schröder,
    chemins de Motzkin, compositions d'entiers, animaux dirigés,~\emph{etc.}
    Nous retrouvons également des opérades déjà connues~: l'opérade magmatique,
    l'opérade commutative associative et l'opérade diassociative.
\end{abstract}

\section{Introduction} \label{sec:Introduction}

Operads are algebraic structures introduced in the 1970s by Boardman and
Vogt~\cite{BV73} and by May~\cite{May72} in the context of algebraic topology.
Informally, an operad is a structure containing operators with~$n$ inputs
and~$1$ output, for all positive integer~$n$. Two operators~$x$ and~$y$
can be composed at $i$th position by grafting the output of~$y$ on the $i$th
input of~$x$. The new operator thus obtained is denoted by~$x \circ_i y$.
In an operad, one can also switch the inputs of an operator~$x$ by letting
a permutation~$\sigma$ act to obtain a new operator denoted by~$x \cdot \sigma$.
One of the main relishes of operads comes from the fact that they offer a
general theory to study in an unifying way different types of algebras,
such as associative algebras and Lie algebras.
\smallskip

In recent years, the importance of operads in combinatorics has continued
to increase and several new operads were defined on combinatorial objects
(see \emph{e.g.},~\cite{Lod01,CL01,Liv06,Cha08}). The structure thereby
added on combinatorial families enables to see these in a new light and
offers original ways to solve some combinatorial problems. For example,
the dendriform operad~\cite{Lod01} is an operad on binary trees and plays
an interesting role for the understanding of the Hopf algebra of Loday-Ronco
of binary trees~\cite{LR98,HNT05}. Besides, this operad is a key ingredient
for the enumeration of intervals of the Tamari lattice~\cite{Cha06,Cha08}.
There is also a very rich link connecting  combinatorial Hopf algebra theory
and operad theory: various constructions produce combinatorial Hopf algebras
from operads~\cite{CL07,LV10}.
\smallskip

In this paper, we propose a new generic method to build combinatorial operads.
The starting point is to pick a monoid~$M$. We then consider the set of
words whose letters are elements of~$M$. The arity of such words are their
length, the composition of two words is expressed from the product of~$M$,
and permutations act on words by permuting letters. In this way, we associate
to any monoid~$M$ an operad denoted by~$\T M$. This construction is rich
from a combinatorial point of view since it allows, by considering suboperads
and quotients of~$\T M$, to get new operads on various combinatorial objects.
\smallskip

This paper is organized as follows. In Section~\ref{sec:Preliminaires},
we recall briefly the basics about set-operads. Section~\ref{sec:Foncteur}
is devoted to the definition of the construction associating an operad
to a monoid and to establish its first properties. We show that this
construction is a functor from the category of monoids to the category of
operads that respects injections and surjections. Finally we apply this construction
in Section~\ref{sec:Exemples} on various monoids and obtain several new
combinatorial (symmetric or not) operads on the following combinatorial
objects: endofunctions, parking functions, packed words, permutations,
planar rooted trees, generalized Dyck paths, Schröder trees, Motzkin paths,
integer compositions, directed animals, and segmented integer compositions.
We conclude by building an operad isomorphic to the diassociative operad~\cite{Lod01}.

\acknowledgements
The author would like to thank Florent Hivert and Jean-Christophe Novelli
for their advice during the preparation of this paper.
This work is based on computer exploration and the author used, for this
purpose, the open-source mathematical software Sage~\cite{Sage} and one
of its extensions, Sage-Combinat~\cite{SageC}.

\section{Preliminaries and notations} \label{sec:Preliminaires}

\subsection{Permutations}

Let us denote by~$[n]$ the set $\{1, \dots, n\}$ and by~$\EnsPermu_n$ the
set of permutations of~$[n]$. Let $\sigma \in \EnsPermu_n$, $\nu \in \EnsPermu_m$,
and $i \in [n]$.  The \emph{substitution} of~$\nu$ into~$\sigma$ is the
permutation $B_i(\sigma, \nu) := \sigma'_1 \dots \sigma'_{i - 1} \nu''_1 \dots \nu''_m \sigma'_{i + 1} \dots \sigma'_n$
where $\sigma'_j := \sigma_j$ if $\sigma_j < \sigma_i$ and $\sigma'_j := \sigma_j + m - 1$
otherwise, and $\nu''_j := \nu_j + \sigma_i - 1$. For instance, one has
$B_4(\textcolor{Bleu}{741}{\bf 5}\textcolor{Bleu}{623}, \textcolor{Rouge}{231})
= \textcolor{Bleu}{941}\textcolor{Rouge}{675}\textcolor{Bleu}{823}$.

\subsection{Operads}

Recall that a \emph{set-operad}, or an \emph{operad} for short, is a set
$\CalP := \biguplus_{n \geq 1} \CalP(n)$ together with \emph{substitution maps}
\begin{equation}
    \circ_i : \CalP(n) \times \CalP(m) \to \CalP(n + m - 1),
    \qquad n, m \geq 1, i \in [n],
\end{equation}
a distinguished element $\Unite \in \CalP(1)$, the \emph{unit} of $\CalP$,
and a \emph{symmetric group action}
\begin{equation}
    \cdot : \CalP(n) \times \EnsPermu_n \to \CalP(n), \qquad n \geq 1.
\end{equation}
The above data has to satisfy the following relations:
\begin{align}
    (x \circ_i y) \circ_{i + j - 1} z = x \circ_i (y \circ_j z),
    \qquad & x \in \CalP(n), y \in \CalP(m), z \in \CalP(k), i \in [n], j \in [m],
    \label{eq:AssocSerie} \\
    (x \circ_i y) \circ_{j + m - 1} z = (x \circ_j z) \circ_i y,
    \qquad & x \in \CalP(n), y \in \CalP(m), z \in \CalP(k), 1 \leq i < j \leq n,
    \label{eq:AssocParallele}\\
    \Unite \circ_1 x = x = x \circ_i \Unite,
    \qquad & x \in \CalP(n), i \in [n], \label{eq:Unite} \\
    (x \cdot \sigma) \circ_i (y \cdot \nu) = \left(x \circ_{\sigma_i} y\right)
    \cdot B_i(\sigma, \nu),
    \qquad & x \in \CalP(n), y \in \CalP(m), \sigma \in \EnsPermu_n, \nu \in \EnsPermu_m,
    i \in [n]. \label{eq:Equivariance}
\end{align}

The \emph{arity} of an element~$x$ of~$\CalP(n)$ is~$n$. Let~$\CalQ$ be
an operad. A map~$\phi : \CalP \to \CalQ$ is an \emph{operad morphism} if
it commutes with substitution maps and symmetric group action and maps
elements of arity~$n$ of~$\CalP$ to elements of arity~$n$ of~$\CalQ$.
A \emph{non-symmetric operad} is an operad without symmetric group action.
The above definitions also work when~$\CalP$ is a~$\EnsNat$-graded vector
space. In this case, the substitution maps~$\circ_i$ are linear maps, and
the symmetric group action is linear on the left.

\section{A combinatorial functor from monoids to operads} \label{sec:Foncteur}

\subsection{The construction}

\subsubsection{Monoids to operads}
Let $(M, \bullet, 1)$ be a monoid. Let us denote by~$\T M$ the set
$\T M := \biguplus_{n \geq 1} \T M(n)$, where for all~$n \geq 1$,
\begin{equation}
    \T M(n) := \left\{(x_1, \dots, x_n) : x_i \in M \mbox{ for all $i \in [n]$}\right\}.
\end{equation}
We endow the set~$\T M$ with maps
\begin{equation} \label{eq:TDomaineSubs}
    \circ_i : \T M(n) \times \T M(m) \to \T M(n + m - 1),
    \qquad n, m \geq 1, i \in [n],
\end{equation}
defined as follows: For all $x \in \T M(n)$, $y \in \T M(m)$, and $i \in [n]$,
we set
\begin{equation} \label{eq:TSub}
    x \circ_i y :=
    (x_1, \; \dots, \; x_{i-1},
    \; x_i \bullet y_1, \; \dots, \; x_i \bullet y_m,
    \; x_{i+1}, \; \dots, \; x_n).
\end{equation}
Let us also set $\Unite := (1)$ as a distinguished element of~$\T M(1)$.
We endow finally each set~$\T M(n)$ with a right action of the symmetric group
\begin{equation}
    \cdot : \T M(n) \times \EnsPermu_n \to \T M(n), \qquad n \geq 1,
\end{equation}
defined as follows: For all $x \in \T M(n)$ and $\sigma \in \EnsPermu_n$,
we set
\begin{equation}
    x \cdot \sigma := \left(x_{\sigma_1}, \dots, x_{\sigma_n}\right).
\end{equation}

The elements of~$\T M$ are words over~$M$ regarded as an alphabet. The
arity of an element~$x$ of~$\T M(n)$, denoted by~$|x|$, is~$n$. For the
sake of readability, we shall denote in some cases an element $(x_1, \dots, x_n)$
of~$\T M(n)$ by its word notation $x_1 \dots x_n$.

\begin{Proposition} \label{prop:TOperade}
    If~$M$ is a monoid, then~$\T M$ is a set-operad.
\end{Proposition}
\begin{proof}
    Let us respectively denote by~$\bullet$ and~$1$ the product and the
    unit of~$M$. First of all, thanks to~(\ref{eq:TDomaineSubs})
    and~(\ref{eq:TSub}), the maps~$\circ_i$ are well-defined and are substitution
    maps of operads. Let us now show that~$\T M$ satisfies~(\ref{eq:AssocSerie}),
    (\ref{eq:AssocParallele}), (\ref{eq:Unite}), and~(\ref{eq:Equivariance}).

    Let $x \in \T M(n)$, $y \in \T M(m)$, $z \in \T M(k)$, $i \in [n]$,
    and $j \in [m]$. We have, using associativity of~$\bullet$, \vspace{-.8em}
    \begin{equation} \begin{split}
        (x \circ_i y) \circ_{i + j - 1} z
        & = (x_1, \; \dots, \; x_{i-1},
            \; x_i \bullet y_1, \; \dots, \; x_i \bullet y_m,
            \; x_{i+1}, \; \dots, \; x_n) \circ_{i + j - 1} z \\
        & = (x_1, \;\dots, \; x_{i-1},
            \; x_i \bullet y_1, \; \dots, \; x_i \bullet y_{j-1},
            \; (x_i \bullet y_j) \bullet z_1, \; \dots, \; (x_i \bullet y_j) \bullet z_k, \\
        & \qquad  x_i \bullet y_{j+1}, \; \dots, \;
            x_i \bullet y_m, \; x_{i+1}, \; \dots, \; x_n) \\
        & = (x_1, \; \dots, \; x_{i-1},
            \; x_i \bullet y_1, \; \dots, \; x_i \bullet y_{j-1},
            \; x_i \bullet (y_j \bullet z_1), \; \dots, \; x_i \bullet (y_j \bullet z_k), \\
        & \qquad x_i \bullet y_{j+1}, \; \dots, \;
            x_i \bullet y_m, \; x_{i+1}, \; \dots, \; x_n) \\
        & = x \circ_i (y_1, \; \dots, \; y_{j-1}, \; y_j \bullet z_1,
            \; \dots, \; y_j \bullet z_k, \; y_{j+1}, \; \dots, \; y_m) \\
        & = x \circ_i (y \circ_j z),
    \end{split} \end{equation}
    showing that~$\circ_i$ satisfies~(\ref{eq:AssocSerie}).

    Let $x \in \T M(n)$, $y \in \T M(m)$, $z \in \T M(k)$,
    and $i < j \in [n]$. We have, \vspace{-.8em}
    \begin{equation} \begin{split}
        (x \circ_j z) \circ_i y
        & = (x_1, \; \dots, \; x_{j-1}, \; x_j \bullet z_1, \; \dots,
            \; x_j \bullet z_k, \; x_{j+1}, \; \dots, \; x_n) \circ_i y, \\
        & = (x_1, \; \dots, \; x_{i-1}, \; x_i \bullet y_1, \; \dots,
            \; x_i \bullet y_m, \; x_{i+1}, \; \dots, \; x_{j-1}, \\
        & \qquad  x_j \bullet z_1, \; \dots, \; x_j \bullet z_k, \; x_{j+1}, \; \dots, \; x_n) \\
        & = (x_1, \; \dots, \; x_{i-1}, \; x_i \bullet y_1, \; \dots,
            \; x_i \bullet y_m, \; x_{i+1}, \; \dots, \; x_n) \circ_{j + m - 1} z \\
        & = (x \circ_i y) \circ_{j + m - 1} z,
    \end{split} \end{equation}
    showing that~$\circ_i$ satisfies~(\ref{eq:AssocParallele}).

    The element~$\Unite$ is the unit of~$\T M$. Indeed, we have
    $\Unite \in \T M(1)$, and, for all $x \in \T M(n)$ and~$i \in [n]$,
    \begin{equation}
        x \circ_i \Unite \enspace = \enspace
        (x_1, \; \dots, \; x_{i-1}, \; x_i \bullet 1, \; x_{i+1}, \; \dots, \; x_n)
        \enspace = \enspace x,
    \end{equation}
    since~$1$ is the unit for~$\bullet$, and
    \begin{equation}
        \Unite \circ_1 x \enspace = \enspace
        (1 \bullet x_1, \; \dots, \; 1 \bullet x_n) \enspace = \enspace
        x,
    \end{equation}
    for the same reason. That shows~(\ref{eq:Unite}).

    Finally, since the symmetric group~$\EnsPermu_n$ acts by permuting the
    letters of a word $(x_1, \dots, x_n)$ of~$\T M(n)$, the maps~$\circ_i$
    and the action~$\cdot$ satisfy together~(\ref{eq:Equivariance}).
\end{proof}

\subsubsection{Monoids morphisms to operads morphisms}
Let~$M$ and~$N$ be two monoids and $\theta : M \to N$ be a monoid morphism.
Let us denote by~$\T \theta$ the map
\begin{equation}
    \T \theta : \T M \to \T N,
\end{equation}
defined for all $(x_1, \dots, x_n) \in \T M(n)$ by
\begin{equation}
    \T \theta\left(x_1, \dots, x_n\right) :=
    \left(\theta(x_1), \dots, \theta(x_n)\right).
\end{equation}

\begin{Proposition} \label{prop:TOperadeMorph}
    If~$M$ and~$N$ are two monoids and $\theta : M \to N$ is a monoid
    morphism, then the map $\T \theta : \T M \to \T N$ is an operad morphism.
\end{Proposition}
\begin{proof}
    Let us respectively denote by~$\bullet_M$ (resp.~$\bullet_N$) and~$1_M$
    (resp.~$1_N$) the product and the unit of~$M$ (resp.~$N$).

    Let $x \in \T M(n)$, $y \in \T M(m)$, and $i \in [n]$.
    Since $\theta$ is a monoid morphism, we have \vspace{-.8em}
    \begin{equation} \begin{split}
        \T \theta (x \circ_i y)
        & = \T \theta\left(x_1, \; \dots, \; x_{i-1}, \; x_i \bullet_M y_1,
            \; \dots, \; x_i \bullet_M y_m, \; x_{i+1}, \; \dots, \; x_n\right) \\
        & = (\theta(x_1), \; \dots, \; \theta(x_{i-1}), \; \theta(x_i \bullet_M y_1),
            \; \dots, \; \theta(x_i \bullet_M y_m), \;
        \theta(x_{i+1}), \; \dots, \; \theta(x_n)) \\
        & = (\theta(x_1), \; \dots, \theta(x_{i-1}), \; \theta(x_i) \bullet_N \theta(y_1),
            \; \dots, \; \theta(x_i) \bullet_N \theta(y_m), \;
            \theta(x_{i+1}), \; \dots, \; \theta(x_n)) \\
        & = \left(\theta(x_1), \; \dots, \; \theta(x_n)) \circ_i (\theta(y_1),
            \; \dots, \; \theta(y_m)\right) \\
        & = \T \theta(x) \circ_i \T \theta(y).
    \end{split} \end{equation}

    Moreover, since~$(1_M)$ is by definition the unit of~$\T M$, we have
    \begin{equation}
        \T \theta\left(1_M\right) \enspace = \enspace
        \left(\theta\left(1_M\right)\right) \enspace = \enspace
        \left(1_N\right).
    \end{equation}

    Finally, since the symmetric group~$\EnsPermu_n$ acts by permuting
    letters, we have for all $x \in \T M(n)$ and $\sigma \in \EnsPermu_n$,
    $\T \theta(x \cdot \sigma) = \T \theta(x) \cdot \sigma$.

    The map~$\T \theta$ satisfies the three required properties
    and hence, since by Proposition~\ref{prop:TOperade}, $\T M$ and $\T N$
    are operads,~$\T \theta$ is an operad morphism.
\end{proof}

\subsection{Properties of the construction}

\begin{Proposition} \label{prop:TInjSur}
    Let~$M$ and~$N$ be two monoids and $\theta : M \to N$ be a monoid morphism.
    If~$\theta$ is injective (resp. surjective), then~$\T \theta$ is
    injective (resp. surjective).
\end{Proposition}
\begin{proof}
    Assume that~$\theta$ is injective and that there are two elements~$x$
    and~$y$ of~$\T M$ such that $\T \theta(x) = \T \theta(y)$. Then,
    \begin{equation}
        \T \theta(x) \enspace = \enspace
        (\theta(x_1), \; \dots, \; \theta(x_n)) \enspace = \enspace
        (\theta(y_1), \; \dots, \; \theta(y_n)) = \T \theta(y),
    \end{equation}
    implying $\theta(x_i) = \theta(y_i)$ for all $i \in [n]$. Since~$\theta$
    is injective, we have $x_i = y_i$ for all $i \in [n]$ and thus,~$x = y$.
    Hence, since~$\T \theta$ is, by Proposition~\ref{prop:TOperadeMorph},
    an operad morphism, it also is an injective operad morphism.

    Assume that~$\theta$ is surjective and let~$y$ be an element of~$\T N(n)$.
    Since~$\theta$ is surjective, there are some elements~$x_i$ of~$M$
    such that $\theta(x_i) = y_i$
    for all $i \in [n]$. We have
    \begin{equation}
        \T \theta(x_1, \; \dots, \; x_n) \enspace = \enspace
        (\theta(x_1), \; \dots, \; \theta(x_n)) \enspace = \enspace
        (y_1, \; \dots, \; y_n).
    \end{equation}
    Hence, since $(x_1, \dots, x_n)$ is by definition an element
    of~$\T M(n)$, and since~$\T \theta$ is, by Proposition~\ref{prop:TOperadeMorph},
    an operad morphism, it also is a surjective operad morphism.
\end{proof}

\begin{Theoreme} \label{thm:TFonct}
    The construction~$\T$ is a functor from the category of monoids with
    monoid morphisms to the category of set-operads with operad morphisms.
    Moreover, $\T$ respects injections and surjections.
\end{Theoreme}
\begin{proof}
    By Proposition~\ref{prop:TOperade},~$\T$ constructs a set-operad from a
    monoid, and by Proposition~\ref{prop:TOperadeMorph}, an operad morphism
    from a monoid morphism.

    Let~$M$ be a monoid, $\theta : M \to M$ be the identity morphism on~$M$,
    and~$x$ be an element of~$\T M(n)$. We have
    \begin{equation}
        \T \theta(x) \enspace = \enspace
        (\theta(x_1), \; \dots, \; \theta(x_n)) \enspace = \enspace
        (x_1, \; \dots, \; x_n) \enspace = \enspace
        x,
    \end{equation}
    showing that~$\T \theta$ is the identity morphism on the operad~$\T M$.

    Let $(L, \bullet_L)$, $(M, \bullet_M)$, and $(N, \bullet_N)$ be three
    monoids, $\theta : L \to M$ and $\omega : M \to N$ be two monoid morphisms,
    and~$x$ be an element of~$\T L(n)$. We have \vspace{-.8em}
    \begin{equation} \begin{split}
        \T(\omega \circ \theta)(x)
        & = \left(\omega\left(\theta(x_1)\right), \; \dots, \;
            \omega\left(\theta(x_n)\right)\right) \\
        & = \T \omega\left(\theta(x_1), \; \dots, \; \theta(x_n)\right) \\
        & = \T \omega\left( \T \theta\left(x_1, \; \dots, \; x_n\right)\right) \\
        & = (\T \omega \circ \T \theta)(x),
    \end{split} \end{equation}
    showing that~$\T$ is compatible with map composition. Hence,~$\T$ is
    a functor, and by Proposition~\ref{prop:TInjSur},~$\T$ also respects
    injections and surjections.
\end{proof}

\section{Some operads obtained by the construction} \label{sec:Exemples}

Through this Section, we consider examples of applications of the functor~$\T$.
We shall mainly consider, given a monoid~$M$, some suboperads of~$\T M$,
symmetric or not, and generated by a finite subset of~$\T M$.

We shall denote by~$\EnsNat$ the additive monoid of integers, and for
all~$\ell \geq 1$, by~$\EnsNat_\ell$ the quotient of~$\EnsNat$ consisting
in the set $\{0, 1, \dots, \ell - 1\}$ with the addition modulo $\ell$ as
the operation of $\EnsNat_\ell$. Note that since~$\T$ is a functor that
respects surjective maps (see Theorem~\ref{thm:TFonct}),~$\T \EnsNat_\ell$
is a quotient operad of~$\T \EnsNat$.

The operads constructed in this Section fit into the diagram of non-symmetric
operads represented in Figure~\ref{fig:DiagrammeOperades}. Table~\ref{tab:Operades}
summarizes some information about these operads.
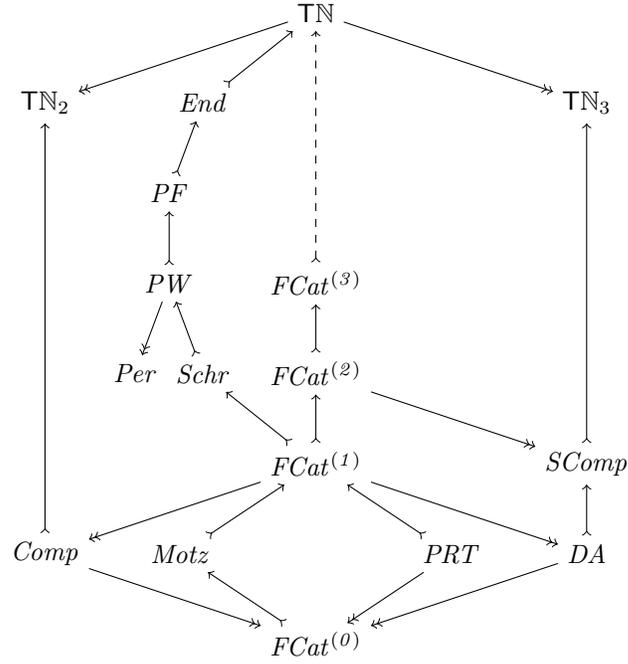
\begin{figure}[ht]
    \centering
    \begin{tikzpicture}[scale=.6]
        \node(TN)at(2,0){$\T \EnsNat$};
        \node(TN2)at(-4,-2){$\T \EnsNat_2$};
        \node(TN3)at(8,-2){$\T \EnsNat_3$};
        \node(End)at(-.5,-2){$\End$};
        \node(FP)at(-1.25,-4){$\FP$};
        \node(MT)at(-1.25,-6){$\MT$};
        \node(Per)at(-2,-8){$\Per$};
        \node(Schr)at(-.5,-8){$\Schr$};
        \node(FCat1)at(2,-10){$\FCat{1}$};
        \node(FCat2)at(2,-8){$\FCat{2}$};
        \node(FCat3)at(2,-6){$\FCat{3}$};
        \node(SComp)at(8,-10){$\SComp$};
        \node(AnD)at(8,-12){$\AnD$};
        \node(APE)at(5,-12){$\APE$};
        \node(Motz)at(-1,-12){$\Motz$};
        \node(Comp)at(-4,-12){$\Comp$};
        \node(FCat0)at(2,-14){$\FCat{0}$};
        \draw[Surjection](TN)--(TN2);
        \draw[Surjection](TN)--(TN3);
        \draw[Injection](End)--(TN);
        \draw[Injection](FP)--(End);
        \draw[Injection](MT)--(FP);
        \draw[Surjection](MT)--(Per);
        \draw[Injection](Schr)--(MT);
        \draw[Injection](FCat1)--(Schr);
        \draw[Injection](FCat1)--(FCat2);
        \draw[Injection](FCat2)--(FCat3);
        \draw[Injection,dashed](FCat3)--(TN);
        \draw[Surjection](FCat2)--(SComp);
        \draw[Injection](SComp)--(TN3);
        \draw[Surjection](FCat1)--(AnD);
        \draw[Injection](AnD)--(SComp);
        \draw[Injection](APE)--(FCat1);
        \draw[Injection](Motz)--(FCat1);
        \draw[Surjection](FCat1)--(Comp);
        \draw[Injection](Comp)--(TN2);
        \draw[Surjection](Comp)--(FCat0);
        \draw[Surjection](APE)--(FCat0);
        \draw[Surjection](AnD)--(FCat0);
        \draw[Injection](FCat0)--(Motz);
    \end{tikzpicture}
    \caption{The diagram of non-symmetric suboperads and quotients of~$\T \EnsNat$.
    Arrows~$\rightarrowtail$ (resp.~$\twoheadrightarrow$) are injective
    (resp. surjective) non-symmetric operad morphisms.}
    \label{fig:DiagrammeOperades}
\end{figure}
\begin{table}[ht]
    \centering
    \begin{tabular}{c|c|c|c}
        Operad     & Generators & First dimensions & Combinatorial objects \\ \hline
        $\End$     & ---          & $1, 4, 27, 256, 3125$       & Endofunctions \\
        $\FP$      & ---          & $1, 3, 16, 125, 1296$       & Parking functions \\
        $\MT$      & ---          & $1, 3, 13, 75, 541$         & Packed words \\
        $\Per$     & ---          & $1, 2, 6, 24, 120$          & Permutations \\
        $\APE$     & $01$  & $1, 1, 2, 5, 14, 42$               & Planar rooted trees \\
        $\FCat{k}$ & $00$, $01$, \dots, $0k$ & Fuss-Catalan numbers & $k$-Dyck paths \\
        $\Schr$    & $00$, $01$, $10$ & $1, 3, 11, 45, 197$     & Schröder trees \\
        $\Motz$    & $00$, $010$      & $1, 1, 2, 4, 9, 21, 51$ & Motzkin paths \\
        $\Comp$    & $00$, $01$       & $1, 2, 4, 8, 16, 32$    & Integer compositions \\
        $\AnD$     & $00$, $01$       & $1, 2, 5, 13, 35, 96$   & Directed animals \\
        $\SComp$   & $00$, $01$, $02$ & $1, 3, 27, 81, 243$     & Segmented integer compositions
    \end{tabular}
    \caption{Generators, first dimensions, and combinatorial objects involved
    in the non-symmetric suboperads and quotients of~$\T \EnsNat$.}
    \label{tab:Operades}
\end{table}

\subsection{Endofunctions, parking functions, packed words, and permutations}

Neither the set of endofunctions nor the set of parking functions, packed
words, and permutations are suboperads of~$\T \EnsNat$. Indeed, one has
the following counterexample:
\begin{equation}
    \textcolor{Bleu}{1}{\bf 2} \circ_2 \textcolor{Rouge}{12} =
    \textcolor{Bleu}{1}\textcolor{Rouge}{34},
\end{equation}
and, even if~$12$ is a permutation,~$134$ is not an endofunction.

Therefore, let us call a word~$u$ a \emph{twisted} endofunction (resp.
parking function, packed word, permutation) if the word
$(u_1 + 1, u_2 + 1, \dots, u_n + 1)$ is an endofunction (resp. parking
function, packed word, permutation). For example, the word~$2300$ is a
twisted endofunction since~$3411$ is an endofunction. Let us denote by~$\End$
(resp.~$\FP$, $\MT$, $\Per$) the set of endofunctions (resp. parking functions,
packed words, permutations). Under this reformulation, one has the following result:
\begin{Proposition} \label{prop:OpEndFPMT}
    The sets~$\End$, $\FP$, and $\MT$ are suboperads of~$\T \EnsNat$.
\end{Proposition}

For example, we have in $\End$ the following substitution:
\begin{equation}
    \textcolor{Bleu}{2}{\bf 1} \textcolor{Bleu}{23} \circ_2 \textcolor{Rouge}{30313} =
        \textcolor{Bleu}{2}\textcolor{Rouge}{41424}\textcolor{Bleu}{23},
\end{equation}
and the following application of the symmetric group action:
\begin{equation}
    11210 \cdot 23514 = 12011.
\end{equation}

Note that~$\End$ is not a finitely generated operad. Indeed, the twisted
endofunctions $u := u_1 \dots u_n$ satisfying $u_i := n - 1$ for all~$i \in [n]$
cannot be obtained by substitutions involving elements of~$\End$ of arity
smaller than~$n$. Similarly,~$\FP$ is not a finitely generated operad since
the twisted parking functions $u := u_1 \dots u_n$ satisfying~$u_i := 0$
for all $i \in [n - 1]$ and $u_n := n - 1$ cannot be obtained by substitutions
involving elements of~$\FP$ of arity smaller than~$n$. However, the
operad~$\MT$ is a finitely generated operad:
\begin{Proposition}  \label{prop:GenerationMT}
    The operad~$\MT$ is the suboperad of~$\T \EnsNat$ generated, as a symmetric
    operad, by the elements~$00$ and~$01$.
\end{Proposition}

Let~$\K$ be a field and let us from now consider that~$\MT$ is an
operad in the category of $\K$-vector spaces, \emph{i.e.},~$\MT$ is the
free $\K$-vector space over the set of twisted packed words with substitution
maps and the right symmetric group action extended by linearity.

Let~$I$ be the free $\K$-vector space over the set of twisted packed words
having multiple occurrences of a same letter.

\begin{Proposition} \label{prop:IdealDeMT}
    The vector space $I$ is an operadic ideal of~$\MT$. Moreover, the operadic
    quotient $\Per := \MT/_I$ is the free vector space over the set of
    twisted permutations.
\end{Proposition}

One has, for all twisted permutations~$x$ and~$y$, the following expression
for the substitution maps in~$\Per$: \vspace{-.8em}
\begin{equation} \label{eq:SubsPartiellePer}
    x \circ_i y =
    \begin{cases}
        x \circ_i y & \mbox{if $x_i = \max x$,} \\
        0_\K    & \mbox{otherwise,}
    \end{cases}
\end{equation}
where~$0_\K$ is the null vector of~$\Per$ and the map~$\circ_i$ in the
right member of~(\ref{eq:SubsPartiellePer}) is the substitution map of~$\MT$.

\subsection{Planar rooted trees}

Let~$\APE$ be the non-symmetric suboperad of~$\T \EnsNat$ generated by~$01$.
One has the following characterization of the elements of~$\APE$:
\begin{Proposition} \label{prop:MotsAPE}
    The elements of~$\APE$ are exactly the words~$x$ on the alphabet~$\EnsNat$
    that satisfy~$x_1 = 0$ and $1 \leq x_{i + 1} \leq x_i + 1$ for all $i \in [|x| - 1]$.
\end{Proposition}

Proposition~\ref{prop:MotsAPE} implies that we can regard the elements
of arity~$n$ of~$\APE$ as planar rooted trees with~$n$ nodes. Indeed,
there is a bijection between words of~$\APE$ and such trees. Given a planar
rooted tree~$T$, one computes an element of~$\APE$ by labelling each
node~$x$ of~$T$ by its depth and then, by reading its labels following
a depth-first traversal of~$T$. Figure~\ref{fig:InterpretationAPE} shows
an example of this bijection.
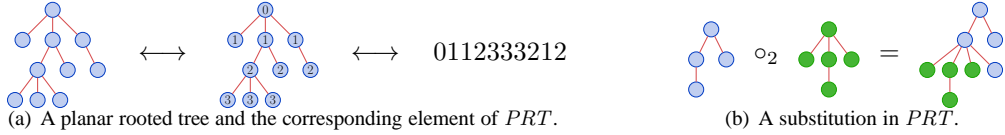
\begin{figure}[ht]
    \centering
    \subfigure[A planar rooted tree and the corresponding element of~$\APE$.]{\makebox[.65\textwidth]{\label{fig:InterpretationAPE}
    \scalebox{.2}{\raisebox{-9em}{\begin{tikzpicture}
        \node[Noeud](1)at(0,0){};
        \node[Noeud](2)at(-2,-2){};
        \node[Noeud](3)at(0,-2){};
        \node[Noeud](4)at(-1,-4){};
        \node[Noeud](5)at(-2.5,-6){};
        \node[Noeud](6)at(-1,-6){};
        \node[Noeud](7)at(0.5,-6){};
        \node[Noeud](8)at(1,-4){};
        \node[Noeud](9)at(2,-2){};
        \node[Noeud](10)at(3,-4){};
        \draw[Arete](1)--(2);
        \draw[Arete](1)--(3);
        \draw[Arete](3)--(4);
        \draw[Arete](4)--(5);
        \draw[Arete](4)--(6);
        \draw[Arete](4)--(7);
        \draw[Arete](3)--(8);
        \draw[Arete](1)--(9);
        \draw[Arete](9)--(10);
    \end{tikzpicture}}}
    $\quad \longleftrightarrow \quad$
    \scalebox{.2}{\raisebox{-9em}{\begin{tikzpicture}
        \node[Noeud](1)at(0,0){\Huge $0$};
        \node[Noeud](2)at(-2,-2){\Huge $1$};
        \node[Noeud](3)at(0,-2){\Huge $1$};
        \node[Noeud](4)at(-1,-4){\Huge $2$};
        \node[Noeud](5)at(-2.5,-6){\Huge $3$};
        \node[Noeud](6)at(-1,-6){\Huge $3$};
        \node[Noeud](7)at(0.5,-6){\Huge $3$};
        \node[Noeud](8)at(1,-4){\Huge $2$};
        \node[Noeud](9)at(2,-2){\Huge $1$};
        \node[Noeud](10)at(3,-4){\Huge $2$};
        \draw[Arete](1)--(2);
        \draw[Arete](1)--(3);
        \draw[Arete](3)--(4);
        \draw[Arete](4)--(5);
        \draw[Arete](4)--(6);
        \draw[Arete](4)--(7);
        \draw[Arete](3)--(8);
        \draw[Arete](1)--(9);
        \draw[Arete](9)--(10);
    \end{tikzpicture}}}
    $\quad \longleftrightarrow \quad 0112333212$}}
    \qquad
    \subfigure[A substitution in~$\APE$.]{\makebox[.25\textwidth]{\label{fig:SubsAPE}
    \scalebox{.2}{\raisebox{-7em}{\begin{tikzpicture}
        \node[Noeud](1)at(0,0){};
        \node[Noeud](2)at(-1,-2){};
        \node[Noeud](3)at(-1,-4){};
        \node[Noeud](4)at(1,-2){};
        \draw[Arete](1)--(2);
        \draw[Arete](2)--(3);
        \draw[Arete](1)--(4);
    \end{tikzpicture}}}
    $\enspace \circ_2 \enspace$
    \scalebox{.2}{\raisebox{-7em}{\begin{tikzpicture}
        \node[Noeud,Marque1](1)at(0,0){};
        \node[Noeud,Marque1](2)at(-1.5,-2){};
        \node[Noeud,Marque1](3)at(0,-2){};
        \node[Noeud,Marque1](4)at(0,-4){};
        \node[Noeud,Marque1](5)at(1.5,-2){};
        \draw[Arete](1)--(2);
        \draw[Arete](1)--(3);
        \draw[Arete](3)--(4);
        \draw[Arete](1)--(5);
    \end{tikzpicture}}}
    $\enspace = \enspace$
    \scalebox{.2}{\raisebox{-9em}{\begin{tikzpicture}
        \node[Noeud](1)at(0,0){};
        \node[Noeud](2)at(-1,-2){};
        \node[Noeud,Marque1](3)at(-3.5,-4){};
        \node[Noeud,Marque1](4)at(-2,-4){};
        \node[Noeud,Marque1](5)at(-2,-6){};
        \node[Noeud,Marque1](6)at(-.5,-4){};
        \node[Noeud](7)at(1,-4){};
        \node[Noeud](8)at(1,-2){};
        \draw[Arete](1)--(2);
        \draw[Arete](2)--(3);
        \draw[Arete](2)--(4);
        \draw[Arete](4)--(5);
        \draw[Arete](2)--(6);
        \draw[Arete](2)--(7);
        \draw[Arete](1)--(8);
    \end{tikzpicture}}}}}
    \caption{Interpretation of the elements and substitution of the non-symmetric
    operad~$\APE$ in terms of planar rooted trees.}
\end{figure}

This bijection offers an alternative way to compute the substitution in~$\APE$:
\begin{Proposition} \label{prop:SubsAPE}
    Let~$S$ and~$T$ be two planar rooted trees and~$s$ be the $i$th visited
    node of~$S$ following its depth-first traversal. The substitution
    $S \circ_i T$ in~$\APE$ amounts to graft the subtrees of the root
    of~$T$ as leftmost sons of~$s$.
\end{Proposition}
Figure~\ref{fig:SubsAPE} shows an example of substitution in~$\APE$.

\begin{Proposition} \label{prop:PresentationAPE}
    The non-symmetric operad~$\APE$ is isomorphic to the free non-symmetric
    operad generated by one element of arity~$2$.
\end{Proposition}
Proposition~\ref{prop:PresentationAPE} also says that~$\APE$ is isomorphic
to the magmatic operad. Hence,~$\APE$ is a realization of the magmatic
operad. Moreover,~$\APE$ can be seen as a planar version of the operad~$\NAP$
of Livernet~\cite{Liv06}.

\subsection{Generalized Dyck paths}

Let~$k \geq 0$ be an integer and~$\FCat{k}$ be the non-symmetric
suboperad of~$\T \EnsNat$ generated by~$00$,~$01$,~\dots,~$0k$. One has
the following characterization of the elements of~$\FCat{k}$:
\begin{Proposition} \label{prop:MotsFCat}
    The elements of~$\FCat{k}$ are exactly the words~$x$ on the alphabet~$\EnsNat$
    that satisfy~$x_1 = 0$ and $0 \leq x_{i + 1} \leq x_i + k$ for all $i \in [|x| - 1]$.
\end{Proposition}

Let us recall that a \emph{$k$-Dyck path} of length~$n$ is a path in~$\EnsNat^2$
connecting the points $(0, 0)$ and $((k + 1)n, 0)$ and consisting in~$n$
\emph{up steps} $(1, k)$ and $kn$ \emph{down steps} $(1, -1)$. It is
well-known that $k$-Dyck paths are enumerated by Fuss-Catalan numbers~\cite{DM47}.
Proposition~\ref{prop:MotsFCat} implies that we can regard the elements
of arity~$n$ of~$\FCat{k}$ as $k$-Dyck paths of length~$n$. Indeed, there
is a bijection between words of~$\FCat{k}$ and such paths. Given a $k$-Dyck
path~$P$, one computes an element~of $\FCat{k}$ by writing, from left to
right, the ordinate of the starting point of each up step of~$P$.
Figure~\ref{fig:BijFCatKDyck} shows an example of this bijection.
\begin{figure}[ht]
    \centering
    \scalebox{.2}{\raisebox{-8em}{\begin{tikzpicture}
        \draw[Grille] (0,0) grid (18,6);
        \node[NoeudDyck](0)at(0,0){};
        \node[NoeudDyck](1)at(1,2){};
        \node[NoeudDyck](2)at(2,1){};
        \node[NoeudDyck](3)at(3,0){};
        \node[NoeudDyck](4)at(4,2){};
        \node[NoeudDyck](5)at(5,4){};
        \node[NoeudDyck](6)at(6,6){};
        \node[NoeudDyck](7)at(7,5){};
        \node[NoeudDyck](8)at(8,4){};
        \node[NoeudDyck](9)at(9,3){};
        \node[NoeudDyck](10)at(10,2){};
        \node[NoeudDyck](11)at(11,1){};
        \node[NoeudDyck](12)at(12,3){};
        \node[NoeudDyck](13)at(13,5){};
        \node[NoeudDyck](14)at(14,4){};
        \node[NoeudDyck](15)at(15,3){};
        \node[NoeudDyck](16)at(16,2){};
        \node[NoeudDyck](17)at(17,1){};
        \node[NoeudDyck](18)at(18,0){};
        \draw[PasDyck](0)--(1);
        \draw[PasDyck](1)--(2);
        \draw[PasDyck](2)--(3);
        \draw[PasDyck](3)--(4);
        \draw[PasDyck](4)--(5);
        \draw[PasDyck](5)--(6);
        \draw[PasDyck](6)--(7);
        \draw[PasDyck](7)--(8);
        \draw[PasDyck](8)--(9);
        \draw[PasDyck](9)--(10);
        \draw[PasDyck](10)--(11);
        \draw[PasDyck](11)--(12);
        \draw[PasDyck](12)--(13);
        \draw[PasDyck](13)--(14);
        \draw[PasDyck](14)--(15);
        \draw[PasDyck](15)--(16);
        \draw[PasDyck](16)--(17);
        \draw[PasDyck](17)--(18);
    \end{tikzpicture}}}
    $\quad \longleftrightarrow \quad$
    \scalebox{.2}{\raisebox{-8em}{\begin{tikzpicture}
        \draw[Grille] (0,0) grid (18,6);
        \node[NoeudDyck](0)at(0,0){};
        \node[NoeudDyck](1)at(1,2){};
        \node[NoeudDyck](2)at(2,1){};
        \node[NoeudDyck](3)at(3,0){};
        \node[NoeudDyck](4)at(4,2){};
        \node[NoeudDyck](5)at(5,4){};
        \node[NoeudDyck](6)at(6,6){};
        \node[NoeudDyck](7)at(7,5){};
        \node[NoeudDyck](8)at(8,4){};
        \node[NoeudDyck](9)at(9,3){};
        \node[NoeudDyck](10)at(10,2){};
        \node[NoeudDyck](11)at(11,1){};
        \node[NoeudDyck](12)at(12,3){};
        \node[NoeudDyck](13)at(13,5){};
        \node[NoeudDyck](14)at(14,4){};
        \node[NoeudDyck](15)at(15,3){};
        \node[NoeudDyck](16)at(16,2){};
        \node[NoeudDyck](17)at(17,1){};
        \node[NoeudDyck](18)at(18,0){};
        \draw[PasDyck](0)--(1);
        \draw[PasDyck](1)--(2);
        \draw[PasDyck](2)--(3);
        \draw[PasDyck](3)--(4);
        \draw[PasDyck](4)--(5);
        \draw[PasDyck](5)--(6);
        \draw[PasDyck](6)--(7);
        \draw[PasDyck](7)--(8);
        \draw[PasDyck](8)--(9);
        \draw[PasDyck](9)--(10);
        \draw[PasDyck](10)--(11);
        \draw[PasDyck](11)--(12);
        \draw[PasDyck](12)--(13);
        \draw[PasDyck](13)--(14);
        \draw[PasDyck](14)--(15);
        \draw[PasDyck](15)--(16);
        \draw[PasDyck](16)--(17);
        \draw[PasDyck](17)--(18);
        \node[above of=0]{\Huge $0$};
        \node[above of=3]{\Huge $0$};
        \node[above of=4]{\Huge $2$};
        \node[above of=5]{\Huge $4$};
        \node[above of=11]{\Huge $1$};
        \node[above of=12]{\Huge $3$};
    \end{tikzpicture}}}
    $\quad \longleftrightarrow \quad 002413$
    \caption{A $2$-Dyck path and the corresponding element of the non-symmetric
    operad~$\FCat{2}$.}
    \label{fig:BijFCatKDyck}
\end{figure}
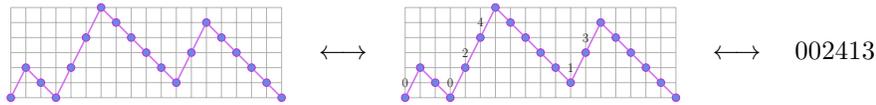

Note that the operad~$\FCat{0}$ is the commutative associative operad.
Next Theorem elucidates the structure of~$\FCat{1}$:
\begin{Theoreme} \label{thm:PresentationFCat1}
    The non-symmetric operad~$\FCat{1}$ is the free non-symmetric operad
    generated by two elements~$\FCatOpA$ and~$\FCatOpB$ of arity~$2$,
    subject to the three relations \vspace{-.8em}

    \begin{minipage}[c]{.4\linewidth}
        \begin{align}
            \FCatOpB \circ_1 \FCatOpB & = \FCatOpB \circ_2 \FCatOpB, \\
            \FCatOpA \circ_1 \FCatOpB & = \FCatOpB \circ_2 \FCatOpA,
        \end{align}
    \end{minipage} \hfill
    \begin{minipage}[c]{.4\linewidth}
        \begin{align}
            \FCatOpA \circ_1 \FCatOpA = \FCatOpA \circ_2 \FCatOpB.
        \end{align}
    \end{minipage}
\end{Theoreme}

\subsection{Schröder trees}

Let~$\Schr$ be the non-symmetric suboperad of~$\T \EnsNat$ generated
by~$00$, $01$, and~$10$. One has the following characterization of the
elements of~$\Schr$:
\begin{Proposition} \label{prop:MotsSchr}
    The elements of~$\Schr$ are exactly the words~$x$ on the alphabet~$\EnsNat$
    that have at least one occurrence of~$0$ and, for all letter $\Lb \geq 1$
    of~$x$, there exists a letter $\La := \Lb - 1$ such that~$x$ has a
    factor $\La u \Lb$ or $\Lb u \La$ where~$u$ is a word consisting in
    letters~$\Lc$ satisfying $\Lc \geq \Lb$.
\end{Proposition}

Recall that a \emph{Schröder tree} is a planar rooted tree such that no
node has exactly one child. The \emph{leaves} of a Schröder tree are the
nodes without children. We call \emph{sector} of a Schröder tree~$T$ a
triple $(x, i, j)$ consisting in a node~$x$ and two adjacent edges~$i$
and~$j$, where~$i$ is immediately on the left of~$j$. Proposition~\ref{prop:MotsSchr}
implies that we can regard the elements of arity~$n$ of~$\Schr$ as Schröder
trees with~$n$ leaves. Indeed, there is a bijection between words of~$\Schr$
and such trees. Given a Schröder tree~$T$, one computes an element of~$\Schr$
by labelling each sector $(x, i, j)$ of~$T$ by the depth of~$x$ and then,
by reading the labels from left to right. Figure~\ref{fig:BijSchrMots} shows
an example of this bijection.
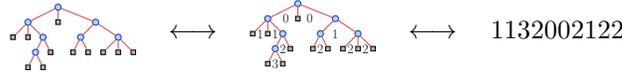
\begin{figure}[ht]
    \centering
    \scalebox{.2}{\raisebox{-6em}{\begin{tikzpicture}
        \node[NoeudSr](n1)at(0,0){};
        \node[NoeudSr](n2)at(-2,-1){};
        \node[NoeudSr](n3)at(-1,-2){};
        \node[NoeudSr](n4)at(-1.5,-3){};
        \node[NoeudSr](n5)at(2.5,-1){};
        \node[NoeudSr](n6)at(1.5,-2){};
        \node[NoeudSr](n7)at(4,-2){};
        \node[Feuille](f1)at(-3,-2){};
        \node[Feuille](f2)at(-2,-2){};
        \node[Feuille](f3)at(-2,-4){};
        \node[Feuille](f4)at(-1,-4){};
        \node[Feuille](f5)at(-.5,-3){};
        \node[Feuille](f6)at(0,-1){};
        \node[Feuille](f7)at(1,-3){};
        \node[Feuille](f8)at(2,-3){};
        \node[Feuille](f9)at(3,-3){};
        \node[Feuille](f10)at(4,-3){};
        \node[Feuille](f11)at(5,-3){};
        \draw[Arete](n1)--(n2);
        \draw[Arete](n2)--(n3);
        \draw[Arete](n3)--(n4);
        \draw[Arete](n1)--(n5);
        \draw[Arete](n5)--(n6);
        \draw[Arete](n5)--(n7);
        \draw[Arete](n1)--(f6);
        \draw[Arete](n2)--(f1);
        \draw[Arete](n2)--(f2);
        \draw[Arete](n3)--(f5);
        \draw[Arete](n4)--(f3);
        \draw[Arete](n4)--(f4);
        \draw[Arete](n6)--(f7);
        \draw[Arete](n6)--(f8);
        \draw[Arete](n7)--(f9);
        \draw[Arete](n7)--(f10);
        \draw[Arete](n7)--(f11);
    \end{tikzpicture}}}
    $\quad \longleftrightarrow \quad$
    \scalebox{.2}{\raisebox{-6em}{\begin{tikzpicture}
        \node[NoeudSr](n1)at(0,0){};
        \node[NoeudSr](n2)at(-2,-1){};
        \node[NoeudSr](n3)at(-1,-2){};
        \node[NoeudSr](n4)at(-1.5,-3){};
        \node[NoeudSr](n5)at(2.5,-1){};
        \node[NoeudSr](n6)at(1.5,-2){};
        \node[NoeudSr](n7)at(4,-2){};
        \node[Feuille](f1)at(-3,-2){};
        \node[Feuille](f2)at(-2,-2){};
        \node[Feuille](f3)at(-2,-4){};
        \node[Feuille](f4)at(-1,-4){};
        \node[Feuille](f5)at(-.5,-3){};
        \node[Feuille](f6)at(0,-1){};
        \node[Feuille](f7)at(1,-3){};
        \node[Feuille](f8)at(2,-3){};
        \node[Feuille](f9)at(3,-3){};
        \node[Feuille](f10)at(4,-3){};
        \node[Feuille](f11)at(5,-3){};
        \draw[Arete](n1)--(n2);
        \draw[Arete](n2)--(n3);
        \draw[Arete](n3)--(n4);
        \draw[Arete](n1)--(n5);
        \draw[Arete](n5)--(n6);
        \draw[Arete](n5)--(n7);
        \draw[Arete](n1)--(f6);
        \draw[Arete](n2)--(f1);
        \draw[Arete](n2)--(f2);
        \draw[Arete](n3)--(f5);
        \draw[Arete](n4)--(f3);
        \draw[Arete](n4)--(f4);
        \draw[Arete](n6)--(f7);
        \draw[Arete](n6)--(f8);
        \draw[Arete](n7)--(f9);
        \draw[Arete](n7)--(f10);
        \draw[Arete](n7)--(f11);
        \node[left of=f2,node distance=.5cm]{\Huge $1$};
        \node[right of=f2,node distance=.5cm]{\Huge $1$};
        \node[right of=f3,node distance=.5cm]{\Huge $3$};
        \node[left of=f5,node distance=.5cm]{\Huge $2$};
        \node[left of=f6,node distance=.8cm]{\Huge $0$};
        \node[right of=f6,node distance=.8cm]{\Huge $0$};
        \node[right of=f7,node distance=.5cm]{\Huge $2$};
        \node[right of=n6,node distance=1cm]{\Huge $1$};
        \node[right of=f9,node distance=.5cm]{\Huge $2$};
        \node[right of=f10,node distance=.5cm]{\Huge $2$};
    \end{tikzpicture}}}
    $\quad \longleftrightarrow \quad 1132002122$
    \caption{A Schröder tree and the corresponding element of the
    non-symmetric operad~$\Schr$.}
    \label{fig:BijSchrMots}
\end{figure}

Let us respectively denote by~$\SchrOpA$,~$\SchrOpB$, and~$\SchrOpC$ the
generators~$00$, $01$, and~$10$ of~$\Schr$.
\begin{Proposition} \label{prop:RelationsGenSchr}
    The generators~$\SchrOpA$, $\SchrOpB$, and~$\SchrOpC$ of~$\Schr$
    are subject, in degree~$2$, to the seven relations \vspace{-.8em}

    \begin{minipage}[c]{.4\linewidth}
        \begin{align}
            \SchrOpA \circ_1 \SchrOpA & = \SchrOpA \circ_2 \SchrOpA, \\
            \SchrOpB \circ_1 \SchrOpC & = \SchrOpC \circ_2 \SchrOpB, \\
            \SchrOpA \circ_1 \SchrOpB & = \SchrOpA \circ_2 \SchrOpC, \\
            \SchrOpB \circ_1 \SchrOpA & = \SchrOpA \circ_2 \SchrOpB,
        \end{align}
    \end{minipage} \hfill
    \begin{minipage}[c]{.4\linewidth}
        \begin{align}
            \SchrOpA \circ_1 \SchrOpC & = \SchrOpC \circ_2 \SchrOpA, \\
            \SchrOpB \circ_1 \SchrOpB & = \SchrOpB \circ_2 \SchrOpA, \\
            \SchrOpC \circ_1 \SchrOpA & = \SchrOpC \circ_2 \SchrOpC.
        \end{align}
    \end{minipage}
\end{Proposition}

\subsection{Motzkin paths}

Let~$\Motz$ be the non-symmetric suboperad of~$\T \EnsNat$ generated
by~$00$ and~$010$. Since~$00$ and~$01$ generate~$\FCat{1}$ and since
$010 = 00 \circ_1 01$, $\Motz$ is a non-symmetric suboperad of~$\FCat{1}$.
One has the following characterization of the elements of~$\Motz$:
\begin{Proposition} \label{prop:ElemMotz}
    The elements of~$\Motz$ are exactly the words~$x$ on the alphabet~$\EnsNat$
    that begin and start by~$0$ and such that $|x_i - x_{i + 1}| \leq 1$
    for all $i \in [|x| - 1]$.
\end{Proposition}

Recall that a Motzkin path of length~$n$ is a path in~$\EnsNat^2$ connecting
the points $(0, 0)$ and $(n, 0)$, and consisting in \emph{up steps} $(1, 1)$,
\emph{down steps} $(1, -1)$, and \emph{stationary steps} $(1, 0)$.
Proposition~\ref{prop:ElemMotz} implies that we can regard the elements
of arity~$n$ of~$\Motz$ as Motzkin paths of length $n - 1$. Indeed there
is a bijection between words of~$\Motz$ and such paths. Given a Motzkin
path~$P$, one computes an element of~$\Motz$ by writing for each point~$p$
of~$P$ the ordinate of~$p$, and then, by reading these labels from left
to right. Figure~\ref{fig:BijMotzMots} shows an example of this bijection.
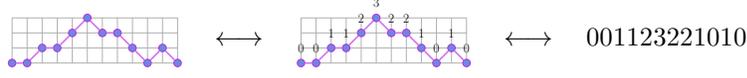
\begin{figure}[ht]
    \centering
    \scalebox{.2}{\raisebox{-4em}{\begin{tikzpicture}
        \draw[Grille] (0,0) grid (11,3);
        \node[NoeudDyck](0)at(0,0){};
        \node[NoeudDyck](1)at(1,0){};
        \node[NoeudDyck](2)at(2,1){};
        \node[NoeudDyck](3)at(3,1){};
        \node[NoeudDyck](4)at(4,2){};
        \node[NoeudDyck](5)at(5,3){};
        \node[NoeudDyck](6)at(6,2){};
        \node[NoeudDyck](7)at(7,2){};
        \node[NoeudDyck](8)at(8,1){};
        \node[NoeudDyck](9)at(9,0){};
        \node[NoeudDyck](10)at(10,1){};
        \node[NoeudDyck](11)at(11,0){};
        \draw[PasDyck](0)--(1);
        \draw[PasDyck](1)--(2);
        \draw[PasDyck](2)--(3);
        \draw[PasDyck](3)--(4);
        \draw[PasDyck](4)--(5);
        \draw[PasDyck](5)--(6);
        \draw[PasDyck](6)--(7);
        \draw[PasDyck](7)--(8);
        \draw[PasDyck](8)--(9);
        \draw[PasDyck](9)--(10);
        \draw[PasDyck](10)--(11);
    \end{tikzpicture}}}
    $\quad \longleftrightarrow \quad$
    \scalebox{.2}{\raisebox{-4em}{\begin{tikzpicture}
        \draw[Grille] (0,0) grid (11,3);
        \node[NoeudDyck](0)at(0,0){};
        \node[NoeudDyck](1)at(1,0){};
        \node[NoeudDyck](2)at(2,1){};
        \node[NoeudDyck](3)at(3,1){};
        \node[NoeudDyck](4)at(4,2){};
        \node[NoeudDyck](5)at(5,3){};
        \node[NoeudDyck](6)at(6,2){};
        \node[NoeudDyck](7)at(7,2){};
        \node[NoeudDyck](8)at(8,1){};
        \node[NoeudDyck](9)at(9,0){};
        \node[NoeudDyck](10)at(10,1){};
        \node[NoeudDyck](11)at(11,0){};
        \draw[PasDyck](0)--(1);
        \draw[PasDyck](1)--(2);
        \draw[PasDyck](2)--(3);
        \draw[PasDyck](3)--(4);
        \draw[PasDyck](4)--(5);
        \draw[PasDyck](5)--(6);
        \draw[PasDyck](6)--(7);
        \draw[PasDyck](7)--(8);
        \draw[PasDyck](8)--(9);
        \draw[PasDyck](9)--(10);
        \draw[PasDyck](10)--(11);
        \node[above of=0]{\Huge $0$};
        \node[above of=1]{\Huge $0$};
        \node[above of=2]{\Huge $1$};
        \node[above of=3]{\Huge $1$};
        \node[above of=4]{\Huge $2$};
        \node[above of=5]{\Huge $3$};
        \node[above of=6]{\Huge $2$};
        \node[above of=7]{\Huge $2$};
        \node[above of=8]{\Huge $1$};
        \node[above of=9]{\Huge $0$};
        \node[above of=10]{\Huge $1$};
        \node[above of=11]{\Huge $0$};
    \end{tikzpicture}}}
    $\quad \longleftrightarrow \quad 001123221010$
    \caption{A Motzkin path and the corresponding element of the
    non-symmetric operad~$\Motz$.}
    \label{fig:BijMotzMots}
\end{figure}

Let us respectively denote by~$\MotzOpA$ and~$\MotzOpB$ the generators~$00$
and~$010$ of~$\Motz$.
\begin{Proposition} \label{prop:RelationsGenMotz}
    The generators~$\MotzOpA$ and~$\MotzOpB$ of~$\Motz$ are subject,
    in degree~$2$, to the four relations \vspace{-.8em}

    \begin{minipage}[c]{.4\linewidth}
        \begin{align}
            \MotzOpA \circ_1 \MotzOpA & = \MotzOpA \circ_2 \MotzOpA, \\
            \MotzOpB \circ_1 \MotzOpA & = \MotzOpA \circ_2 \MotzOpB,
        \end{align}
    \end{minipage} \hfill
    \begin{minipage}[c]{.4\linewidth}
        \begin{align}
            \MotzOpA \circ_1 \MotzOpB & = \MotzOpB \circ_3 \MotzOpA, \\
            \MotzOpB \circ_1 \MotzOpB & = \MotzOpB \circ_3 \MotzOpB.
        \end{align}
    \end{minipage}
\end{Proposition}

\subsection{Integer compositions}

Let~$\Comp$ be the non-symmetric suboperad of~$\T \EnsNat_2$ generated
by~$00$ and~$01$. Since~$\FCat{1}$ is the non-symmetric suboperad of~$\T \EnsNat$
generated by~$00$ and~$01$, and since~$\T \EnsNat_2$ is a quotient of~$\T \EnsNat$,
$\Comp$ is a quotient of~$\FCat{1}$. One has the following characterization
of the elements of~$\Comp$~:
\begin{Proposition} \label{prop:ElemComp}
    The elements of~$\Comp$ are exactly the words on the alphabet~$\{0, 1\}$
    that begin by~$0$.
\end{Proposition}

Proposition~\ref{prop:ElemComp} implies that we can regard the elements
of arity~$n$ of~$\Comp$ as integer compositions of~$n$. Indeed, there is a
bijection between words of~$\Comp$ and integer compositions. Given a
composition $C := (C_1, C_2, \dots, C_\ell)$, one computes the following
element of~$\Comp$:
\begin{equation}
    01^{C_1 - 1} 01^{C_2 - 1} \dots 01^{C_\ell - 1}.
\end{equation}

Encoding integer compositions by ribbon diagrams offers an alternative
way to compute the substitution in~$\Comp$:
\begin{Proposition} \label{prop:SubsComp}
    Let~$C$ and~$D$ be two ribbon diagrams, $i$ be an integer, and~$c$ be
    the $i$th visited box of~$C$ by scanning it from up to down and from
    left to right. Then, the substitution $C \circ_i D$ in~$\Comp$
    returns to replace~$c$ by~$D$ if~$c$ is the upper box of its column,
    or to replace~$c$ by the transpose ribbon diagram of~$D$ otherwise.
\end{Proposition}
Figure~\ref{fig:SubsComp} shows two examples of substitution in~$\Comp$.
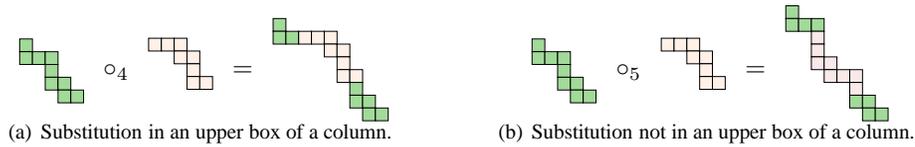
\begin{figure}[ht]
    \centering
    \subfigure[Substitution in an upper box of a column.]{\makebox[.45\textwidth]{
    \scalebox{.17}{\raisebox{-6em}{\begin{tikzpicture}
        \node[Boite]at(0,0){};
        \node[Boite]at(0,-1){};
        \node[Boite]at(1,-1){};
        \node[Boite]at(2,-1){};
        \node[Boite]at(2,-2){};
        \node[Boite]at(2,-3){};
        \node[Boite]at(3,-3){};
        \node[Boite]at(3,-4){};
        \node[Boite]at(4,-4){};
    \end{tikzpicture}}}
    $\enspace \circ_4 \enspace$
    \scalebox{.17}{\raisebox{-3em}{\begin{tikzpicture}
        \node[Boite2]at(0,0){};
        \node[Boite2]at(1,0){};
        \node[Boite2]at(2,0){};
        \node[Boite2]at(2,-1){};
        \node[Boite2]at(3,-1){};
        \node[Boite2]at(3,-2){};
        \node[Boite2]at(3,-3){};
        \node[Boite2]at(4,-3){};
    \end{tikzpicture}}}
    $\enspace = \enspace$
    \scalebox{.17}{\raisebox{-10em}{\begin{tikzpicture}
        \node[Boite]at(0,0){};
        \node[Boite]at(0,-1){};
        \node[Boite]at(1,-1){};
        \node[Boite2]at(2,-1){};
        \node[Boite2]at(3,-1){};
        \node[Boite2]at(4,-1){};
        \node[Boite2]at(4,-2){};
        \node[Boite2]at(5,-2){};
        \node[Boite2]at(5,-3){};
        \node[Boite2]at(5,-4){};
        \node[Boite2]at(6,-4){};
        \node[Boite]at(6,-5){};
        \node[Boite]at(6,-6){};
        \node[Boite]at(7,-6){};
        \node[Boite]at(7,-7){};
        \node[Boite]at(8,-7){};
    \end{tikzpicture}}}}}
    \subfigure[Substitution not in an upper box of a column.]{\makebox[.45\textwidth]{
    \scalebox{.17}{\raisebox{-6em}{\begin{tikzpicture}
        \node[Boite]at(0,0){};
        \node[Boite]at(0,-1){};
        \node[Boite]at(1,-1){};
        \node[Boite]at(2,-1){};
        \node[Boite]at(2,-2){};
        \node[Boite]at(2,-3){};
        \node[Boite]at(3,-3){};
        \node[Boite]at(3,-4){};
        \node[Boite]at(4,-4){};
    \end{tikzpicture}}}
    $\enspace \circ_5 \enspace$
    \scalebox{.17}{\raisebox{-3em}{\begin{tikzpicture}
        \node[Boite2]at(0,0){};
        \node[Boite2]at(1,0){};
        \node[Boite2]at(2,0){};
        \node[Boite2]at(2,-1){};
        \node[Boite2]at(3,-1){};
        \node[Boite2]at(3,-2){};
        \node[Boite2]at(3,-3){};
        \node[Boite2]at(4,-3){};
    \end{tikzpicture}}}
    $\enspace = \enspace$
    \scalebox{.17}{\raisebox{-10em}{\begin{tikzpicture}
        \node[Boite]at(0,0){};
        \node[Boite]at(0,-1){};
        \node[Boite]at(1,-1){};
        \node[Boite]at(2,-1){};
        \node[Boite3]at(2,-2){};
        \node[Boite3]at(2,-3){};
        \node[Boite3]at(2,-4){};
        \node[Boite3]at(3,-4){};
        \node[Boite3]at(3,-5){};
        \node[Boite3]at(4,-5){};
        \node[Boite3]at(5,-5){};
        \node[Boite3]at(5,-6){};
        \node[Boite]at(5,-7){};
        \node[Boite]at(6,-7){};
        \node[Boite]at(6,-8){};
        \node[Boite]at(7,-8){};
    \end{tikzpicture}}}}}
    \caption{Two examples of substitutions in the non-symmetric operad~$\Comp$.}
    \label{fig:SubsComp}
\end{figure}

\begin{Theoreme} \label{thm:PresentationComp}
    The non-symmetric operad~$\Comp$ is the free non-symmetric operad generated
    by two elements~$\CompOpA$ and~$\CompOpB$ of arity~$2$, subject to the
    four relations \vspace{-.8em}

    \begin{minipage}[c]{.4\linewidth}
        \begin{align}
            \CompOpA \circ_1 \CompOpA & = \CompOpA \circ_2 \CompOpA, \\
            \CompOpB \circ_1 \CompOpA & = \CompOpA \circ_2 \CompOpB,
        \end{align}
    \end{minipage} \hfill
    \begin{minipage}[c]{.4\linewidth}
        \begin{align}
            \CompOpB \circ_1 \CompOpB & = \CompOpB \circ_2 \CompOpA, \\
            \CompOpA \circ_1 \CompOpB & = \CompOpB \circ_2 \CompOpB.
        \end{align}
    \end{minipage}
\end{Theoreme}

\subsection{Directed animals}

Let~$\AnD$ be the non-symmetric suboperad of~$\T \EnsNat_3$ generated
by~$00$ and~$01$. Since~$\FCat{1}$ is the non-symmetric suboperad of~$\T \EnsNat$
generated by~$00$ and~$01$, and since~$\T \EnsNat_3$ is a quotient of~$\T \EnsNat$,
$\AnD$ is a quotient of~$\FCat{1}$.

From now, we shall represent by~$-1$ the element~$2$ of~$\EnsNat_3$. With
this encoding, let
\begin{equation}
    \phi : \{-1, 0, 1\}^n \to \{-1, 0, 1\}^{n - 1},
\end{equation}
be the map defined for all $\La, \Lb \in \{-1, 0, -1\}$ by
\begin{equation}
    \phi(\La) := \epsilon \qquad \mbox{ and } \qquad
    \phi(\La \cdot \Lb \cdot u) := (\Lb - \La) \cdot \phi(u).
\end{equation}
For example, the element $x := 011220201$ of~$\AnD$ is represented by the
word $x' := 011-\!\!1-\!\!10-\!\!101$ and we have $\phi(x') = 10101-\!\!111$.

\begin{Proposition} \label{prop:BijAnDPrefMotz}
    By interpreting letters~$-1$ (resp.~$0$, $1$) as down (resp. stationary,
    up) steps, the map~$\phi$ induces a bijection between the elements
    of~$\AnD$ of arity~$n$ and the prefixes of Motzkin paths
    of length~$n - 1$.
\end{Proposition}

Recall that a \emph{directed animal} is a subset~$A$ of~$\EnsNat^2$ such
that $(0, 0) \in A$ and $(i, j) \in A$ with~$i \geq 1$ or~$j \geq 1$
implies $(i - 1, j) \in A$ or $(i, j - 1) \in A$. Using a bijection of
Gouyou-Beauchamps and Viennot~\cite{GBV88} between directed animals of
size~$n$ and prefixes of Motzkin paths of length~$n - 1$, one obtains,
by Proposition~\ref{prop:BijAnDPrefMotz}, a bijection between the elements
of~$\AnD$ of arity~$n$ and directed animals of size~$n$. Hence,~$\AnD$ is
a non-symmetric operad on directed animals.

\subsection{Segmented integer compositions}

Let~$\SComp$ be the non-symmetric suboperad of~$\T \EnsNat_3$ generated
by~$00$, $01$, and~$02$. Since~$\FCat{2}$ is the non-symmetric suboperad
of~$\T \EnsNat$ generated by~$00$, $01$, and~$02$, and since~$\T \EnsNat_3$
is a quotient of~$\T \EnsNat$, $\SComp$ is a quotient of~$\FCat{2}$. One
has the following characterization of the elements of~$\SComp$:
\begin{Proposition} \label{prop:ElemSComp}
    The elements of~$\SComp$ are exactly the words on the alphabet~$\{0, 1, 2\}$
    that begin by~$0$.
\end{Proposition}

Recall that a \emph{segmented integer composition} of~$n$ is a sequence
$(S_1, \dots, S_\ell)$ of integers compositions such that~$S_1$ is an integer
composition of~$n_1$, \dots, $S_\ell$ is an integer composition of~$n_\ell$,
and $n_1 + \dots + n_\ell = n$. Proposition~\ref{prop:ElemSComp} implies
that we can regard the elements of arity~$n$ of~$\SComp$ as segmented integer
compositions of~$n$. Indeed, there is a bijection between words of~$\SComp$
and segmented compositions since there are~$3^{n - 1}$ segmented compositions
of~$n$ and also, by the above Proposition, $3^{n - 1}$ elements of~$\SComp$
of arity~$n$.

\subsection{The diassociative operad}

Let~$M$ be the submonoid of the multiplicative monoid restricted to the
set~$\{0, 1\}$. Let~$\DD$ be the non-symmetric suboperad~of $\T M$ generated
by~$01$ and~$10$. One has the following characterization of the elements of~$\DD$:
\begin{Proposition} \label{prop:ElemD}
    The elements of~$\DD$ are exactly the words on the alphabet~$\{0, 1\}$
    that contain exactly one~$1$.
\end{Proposition}

Recall that the \emph{diassociative operad}~\cite{Lod01}~$\Dias$ is the
non-symmetric operad generated by two elements~$\dashv$ and~$\vdash$ of
arity~$2$, subject only to the relations \vspace{-.8em}

\begin{minipage}[c]{.4\linewidth}
    \begin{align}
        \dashv \circ_1 \dashv \enspace = \enspace \dashv \circ_2 \dashv &
            \enspace = \enspace \dashv \circ_2 \vdash, \\
        \vdash \circ_2 \vdash \enspace = \enspace \vdash \circ_1 \vdash &
            \enspace = \enspace \vdash \circ_1 \dashv,
    \end{align}
\end{minipage} \hfill
\begin{minipage}[c]{.4\linewidth}
    \begin{equation}
        \dashv \circ_1 \vdash \enspace = \enspace \vdash \circ_2 \dashv.
    \end{equation}
\end{minipage}

\begin{Proposition} \label{prop:IsoDiasD}
    The non-symmetric operads~$\DD$ and~$\Dias$ are isomorphic. The map
    $\phi : \Dias \to \DD$ defined by $\phi(\dashv) := 10$ and
    $\phi(\vdash) := 01$ is an isomorphism.
\end{Proposition}
Proposition~\ref{prop:IsoDiasD} also says that~$\DD$ is a realization
of the diassociative operad.

\bibliographystyle{plain}
\bibliography{Bibliographie}

\end{document}